\documentclass[11pt,letterpaper]{article}

\usepackage[margin=1in]{geometry}                                                    
\usepackage{amsmath,amsfonts,amssymb,color}
\usepackage{tikz}
\usepackage{pmat}
\usepackage{amsthm}
\usepackage{empheq}

\newtheorem{problem}{Problem}

\newtheorem{definition}{Definition}
\newtheorem{theorem}{Theorem}
\newtheorem{lemma}{Lemma}
\newtheorem{proposition}{Proposition}

\newtheorem{remark}{Remark}
\newtheorem{assumption}{Assumption}

\usepackage{epsfig}

\usepackage{algorithm}
\usepackage{algorithmic}
\usepackage{epsfig}

\usepackage{array}
\usepackage{multirow}
\usepackage{epstopdf}
\usepackage{tikz}
\usepackage{relsize}
\usepackage[pdfpagelabels,pdfpagemode=None,breaklinks=true]{hyperref}\hypersetup{
colorlinks = true,
citecolor = blue}
\usetikzlibrary{shapes,snakes}

\usepackage{cite}

\DeclareMathAlphabet{\mathcal}{OMS}{cmsy}{m}{n} 

\title{\LARGE \bf
Distributed Functional Observers for LTI Systems
}

\author{Aritra~Mitra~and~Shreyas Sundaram
\thanks{This research was supported in part by NSF grant CMMI-1635014. The authors are with the School of Electrical and Computer Engineering at Purdue University. Email: {\tt \{mitra14,sundara2\}@purdue.edu}.}%
}

\begin{document}
\maketitle

\begin{abstract}
We study the problem of designing distributed functional observers for LTI systems. Specifically, we consider a setting consisting of a state vector that evolves over time according to a dynamical process.  A set of nodes distributed over a communication network wish to collaboratively estimate certain functions of the state.  We first show that classical existence conditions for the design of centralized functional observers do not directly translate to the distributed setting, due to the coupling that exists between the dynamics of the functions of interest and the diverse measurements at the various nodes. Accordingly, we design transformations that reveal such couplings and identify portions of the corresponding dynamics that are locally detectable at each sensor node. We provide sufficient conditions on the network, along with state estimate update and exchange rules for each node, that guarantee asymptotic reconstruction of the functions at each sensor node. 
\end{abstract}
\section{Introduction}
Consider the discrete-time linear time-invariant dynamical system
\begin{equation}
\mathbf{x}[k+1] = \mathbf{Ax}[k],
\label{eqn:plant}
\end{equation}
where $k \in \mathbb{N}$ is the discrete-time index, $\mathbf{x}[k] \in {\mathbb{R}}^n$ is the state vector and  $\mathbf{A} \in {\mathbb{R}}^{ n \times n} $ is the system matrix. The state of the system is monitored by a network of $N$ sensors, each of which receives a partial measurement of the state at every time-step. Specifically, the $i$-th sensor has access to a measurement of the state, given by
\begin{equation}
\mathbf{y}_{i}[k]=\mathbf{C}_i\mathbf{x}[k],
\label{eqn:Obsmodel}
\end{equation}
where $\mathbf{y}_{i}[k] \in {\mathbb{R}}^{r_i}$ and $\mathbf{C}_i \in {\mathbb{R}}^{r_i \times n}$. We use $\mathbf{y}[k]={\begin{bmatrix}\mathbf{y}^T_{1}[k] & \cdots & \mathbf{y}^T_{N}[k]\end{bmatrix}}^T$ to represent the collective measurement vector, and $\mathbf{C}={\begin{bmatrix}\mathbf{C}^T_{1} & \cdots & \mathbf{C}^T_{N}\end{bmatrix}}^T$ to denote the collection of the sensor observation matrices. These sensors are represented as nodes of an underlying directed communication graph which governs the information flow between the sensors.  
Each node is capable of exchanging information with its neighbors and performing computational tasks. The goal of each node is to estimate $\boldsymbol{\psi}[k]$, where 
\begin{equation}
\boldsymbol{\psi}[k]=\mathbf{L}\mathbf{x}[k].
\label{eq:funcdef}
\end{equation}
Here, $\mathbf{L} \in \mathbb{R}^{r \times n}$ is a full row-rank matrix (without loss of generality); hence $\boldsymbol{\psi}[k]$ represents $r$ linearly independent functions of the state.\footnote{When $\mathbf{L}$ is the identity matrix (or more generally a square non-singular matrix), the corresponding problem becomes the well-explored distributed state estimation problem. For work on distributed Kalman filtering, see \cite{olfati1,olfati3,infinite1}, and for literature on distributed state observers refer to \cite{Khanobs2,park,wang,mitra2016}. The problem studied in this paper is in fact a generalization of the distributed state estimation problem.} While there is a rich body of literature that looks at the centralized functional observer design problem (see\cite{darouach,trinh}), we are unaware of any work that investigates the distributed counterpart based on the generic model considered in this paper. In \cite{funcdist}, the authors develop a partially distributed functional observer scheme for coupled interconnected LTI systems where each sub-system maintains an observer for estimating functions of the state corresponding to that particular sub-system. However, the model and problem formulation in \cite{funcdist} differs from the one considered in this paper. A key point of difference is that in \cite{funcdist}, the impact of the underlying communication graph does not play a role in the design strategy, whereas our methodology focuses on analyzing the relationship between the system dynamics and the network topology. 

The main contribution of this paper is a distributed algorithm that guarantees asymptotic reconstruction of $\boldsymbol{\psi}[k]$ at each sensor node under certain conditions on the system dynamics and network topology.

\section{System Model}
\label{sec:sysmod}
\subsection{Notation}
\label{sec:notation}
A directed graph is denoted by $\mathcal{G} =(\mathcal{V},\mathcal{E})$, where $\mathcal{V} =\{1, \cdots, N\}$ is the set of nodes and $\mathcal{E} \subseteq \mathcal{V} \times \mathcal{V} $ represents the edges. An edge from node $j$ to node $i$, denoted by $(j,i)$, implies that node $j$ can transmit information to node $i$. The neighborhood of the $i$-th node is defined as $\mathcal{N}_i \triangleq \{i\} \cup \{j\,|\,(j,i) \in \mathcal{E} \}.$  The notation $|\mathcal{V}|$ is used to denote the cardinality of a set $\mathcal{V}$. For a set $\{\mathbf{A}_1, \cdots, \mathbf{A}_n\}$ of matrices, we use $ 
diag(\mathbf{A}_1, \cdots, \mathbf{A}_n)$ to refer to a block-diagonal matrix with matrix $\mathbf{A}_i$ as the $i$-th block-diagonal entry. For a set $\mathcal{S}=\{s_1,\cdots, s_p\} \subseteq \{1,\cdots, N\}$, and a matrix $\mathbf{C}={\begin{bmatrix} \mathbf{C}^T_{1} & \cdots &  \mathbf{C}^T_{N}\end{bmatrix} }^T$, we define $\mathbf{C}_{\mathcal{S}}\triangleq {\begin{bmatrix}\mathbf{C}^{T}_{s_1} & \cdots & \mathbf{C}^{T}_{s_p}\end{bmatrix}}^T$. Given a matrix $\mathbf{A}$, we use $\mathcal{R}(\mathbf{A})$ to denote the row-space of $\mathbf{A}$, $sp(\mathbf{A})$ to denote the spectrum of $\mathbf{A}$ and $\mathbf{A}^{\dagger}$ to refer to its Moore Penrose inverse. We use $\mathbf{I}_{r}$ to indicate an identity matrix of dimension $r \times r$.

\subsection{Problem Formulation}
\label{sec:prob_form}
Consider the LTI system given by \eqref{eqn:plant}, the measurement model specified by \eqref{eqn:Obsmodel}, the functions of interest described by \eqref{eq:funcdef}, and a predefined directed communication graph $\mathcal{G=(V,E)}$ where $\mathcal{V}$ represents the set of $N$ nodes (or sensors). Let $\hat{\boldsymbol{\psi}}_i[k]$ denote the estimate of $\boldsymbol{\psi}[k]$ maintained by node $i$, which it updates at each time-step $k$ based on information received from its neighbors and its local measurements (if any). Given this setting, the problem studied in this paper is formally stated as follows.
\begin{problem} [Distributed Functional  Estimation Problem]
For the model specified by equations \eqref{eqn:plant}, \eqref{eqn:Obsmodel}, \eqref{eq:funcdef} and a predefined communication graph $\mathcal{G}$, design a distributed algorithm that  achieves $\lim_{k\to\infty} ||\hat{\boldsymbol{\psi}}_i[k]-\boldsymbol{\psi}[k]||=0, \forall i \in \{1, \cdots, N\}$.
\end{problem}
\begin{remark}
As long the system is globally detectable, i.e., the pair $(\mathbf{A,C})$ is detectable, and the communication graph $\mathcal{G}$ is strongly connected, a trivial way to solve Problem $1$ is to reconstruct the entire state $\mathbf{x}[k]$ at every sensor node based on any of the existing distributed state estimation techniques in \cite{park,wang,mitra2016}.
Our goal in this paper will be to design observers that are in general of order \footnote{By the order of an observer at a given sensor node, we refer to the dimension of the portion of the state that is dynamically estimated at that node, i.e., the portion that is not obtained directly from the measurements of the node under consideration.}  smaller than the dimension of the state $\mathbf{x}[k]$. For more on this issue, refer to Remark \ref{rem:order}.
\end{remark}

A distributed algorithm that solves Problem $1$ will be called a \textit{distributed functional observer}. 
Note that in general it may not be possible for a given node $i$ to estimate $\boldsymbol{\psi}[k]$ by solely relying on its own measurements,\footnote{This is precisely the case when $\boldsymbol{\psi}[k]$ is a linear function of some states that are undetectable with respect to the measurements of node $i$; for related notions of `Functional Observability', see \cite{func1,func2}.} thereby dictating the need to exchange information with its neighbors. However, this exchange of information is restricted by the underlying communication graph $\mathcal{G}$. In addition to the aforementioned challenges, in the subsequent section we show that existing results/techniques for the centralized version of the problem are not directly applicable to the problem under consideration. 
\section{Motivation}
\label{sec:motivation}
The purpose of this section is to illustrate that classical existence conditions for the design of centralized functional observers (of a given order) do not generally hold in a distributed setting, and in the process, motivate our present work. To this end, we first recall the following necessary and sufficient conditions set forth by Darouach \cite{darouach} for the existence of a centralized functional observer of order $r$, where $r=\textrm{rank}\hspace{1mm}\mathbf{L}$:
\begin{itemize}
\item[(i)] \begin{equation}
\textrm{rank}\begin{bmatrix}\mathbf{LA}\\ \mathbf{CA} \\ \mathbf{C} \\ \mathbf{L} \end{bmatrix}=\textrm{rank}\begin{bmatrix}\mathbf{CA}\\\mathbf{C}\\\mathbf{L}\end{bmatrix},
\label{eq:Darouach_rank}
\end{equation} 

\item[(ii)] 
\begin{equation}
\textrm{rank}\begin{bmatrix}s\mathbf{L}-\mathbf{L}\mathbf{A}\\
\mathbf{CA}\\
\mathbf{C}
\end{bmatrix}
=\textrm{rank}\begin{bmatrix}\mathbf{CA}\\\mathbf{C}\\\mathbf{L}\end{bmatrix}, \hspace{2mm}\forall s\in \mathbb{C}, \hspace{2mm} |s| \geq 1.
\label{eq:Darouach_detect}
\end{equation}
\end{itemize} 
Next, consider the following model, where the system is monitored by a network of nodes, as depicted by Figure \ref{fig:example}:
\begin{eqnarray}
\underbrace{\begin{bmatrix}{x}^{(1)}[k+1] \\{x}^{(2)}[k+1]\end{bmatrix}}_{\mathbf{x}[k+1]}=\underbrace{\begin{bmatrix}\frac{1}{2}&2\\0&3\end{bmatrix}}_{\mathbf{A}}\underbrace{\begin{bmatrix}x^{(1)}[k]\\{x}^{(2)}[k]\end{bmatrix}}_{\mathbf{x}[k]},\nonumber\\
\mathbf{L}=\begin{bmatrix}1&0\end{bmatrix}, \mathbf{C}_1=\begin{bmatrix}0&1\end{bmatrix}, \mathbf{C}_2=\mathbf{C}_3=0.
\label{eq:example_model}
\end{eqnarray}
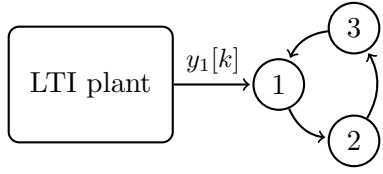
\begin{figure}[t]
\begin{center}
\begin{tikzpicture}
[->,shorten >=1pt,scale=.5, minimum size=10pt, auto=center, node distance=3cm,
  thick, every node/.style={circle, draw=black, thick},]
\tikzstyle{block} = [rectangle, draw, 
    text width=5em, text centered, rounded corners, minimum height=4em];
\node [block]  at (0,3.5) (plant) {LTI plant};

\node [circle, draw](n1) at (5,3.5)  (1)  {1};
\node [circle, draw](n2) at (7,2)   (2)  {2};
\node [circle, draw](n3) at (7,5)   (3)  {3};

\path[every node/.style={font=\sffamily\small}]
(plant)
             edge[] node [above] {$y_1[k]$} (1)
             
    (1)
        edge [bend right] node [] {} (2)
        
    (2)
        edge [bend right] node [] {} (3)
        
(3)
        edge [bend right] node [] {} (1);      
\end{tikzpicture}
\end{center}
\caption{Example for illustrating Proposition \ref{prop:example}.}
\label{fig:example}
\end{figure}
The objective is to asymptotically estimate $x^{(1)}[k]$ at each of the three sensor nodes. It is easy to verify that the necessary and sufficient conditions (equations \eqref{eq:Darouach_rank} and \eqref{eq:Darouach_detect})  for the existence of a centralized $1$st order functional observer are satisfied by the model \eqref{eq:example_model} with $\mathbf{C}={\begin{bmatrix} \mathbf{C}^T_1&\mathbf{C}^T_2&\mathbf{C}^T_3\end{bmatrix}}^T$. At this point, the natural inclination is to ascertain whether it is possible for each sensor node to asymptotically estimate $x^{(1)}[k]$ via 1st order estimators. To formally answer this question, we need to impart some structure to the distributed observers under consideration. To this end, consider distributed observers of the form\footnote{The choice of this observer structure is inspired by the fact that standard distributed state observers existing in literature are essentially of this form.}
\begin{equation}
{\hat{x}}^{(1)}_{i}[k+1]=\alpha_i \sum_{j \in\mathcal{N}_i}w_{ij}{\hat{x}}^{(1)}_{j}[k]+\beta_i\sum_{j\in\mathcal{N}_i}y_j[k],
\label{eqn:distobsexample}
\end{equation}
where ${\hat{x}}^{(1)}_{i}[k]$ is the estimate of state $x^{(1)}[k]$ maintained by node $i$ at time-step $k$, and $\alpha_i$, $\beta_i$, $w_{ij}$ are free design parameters at node $i$. Moreover, the weights $w_{ij}$ are non-negative and satisfy $\sum_{j\in\mathcal{N}_i}w_{ij}=1, i\in\{1,2,3\}$. We have the following simple result.
\begin{proposition}
For the model given by \eqref{eq:example_model}, and the corresponding network depicted by Figure \ref{fig:example}, it is impossible for node $3$ to  estimate the function $\mathbf{Lx}[k]=x^{(1)}[k]$ using a 1st order observer that has structure given by \eqref{eqn:distobsexample}.
\label{prop:example}
\end{proposition}
\begin{proof} Let $e_i[k]={\hat{x}}^{(1)}_{i}[k]-x^{(1)}[k]$ denote the error in estimation of state $x^{(1)}[k]$ at node $i$. Based on \eqref{eq:example_model}, we have $
x^{(1)}[k+1]=\alpha x^{(1)}[k]+\beta y_1[k]
$
where $\alpha=\frac{1}{2}$, $\beta=2$ and $y_1[k]=x^{(2)}[k]$. Then, using \eqref{eqn:distobsexample} and some straightforward algebra, we obtain the following error dynamics:
\begin{eqnarray}
\underbrace{\begin{bmatrix}e_1[k+1]\\e_2[k+1]\\e_3[k+1]\end{bmatrix}}_{\mathbf{e}[k+1]}&=\underbrace{\begin{bmatrix}\alpha_1w_{11}&0&\alpha_1w_{13}\\\alpha_2w_{21}&\alpha_2w_{22}&0\\0&\alpha_3w_{32}&\alpha_3w_{33}\end{bmatrix}}_{\mathbf{M}}\underbrace{\begin{bmatrix}e_1[k]\\e_2[k]\\e_3[k]\end{bmatrix}}_{\mathbf{e}[k]}\\
&+\underbrace{\begin{bmatrix}(\alpha_1-\alpha)\\(\alpha_2-\alpha)\\(\alpha_3-\alpha)
\end{bmatrix}}_{\mathbf{B}_1}x^{(1)}[k]+\underbrace{\begin{bmatrix}(\beta_1-\beta)\\(\beta_2-\beta)\\-\beta
\end{bmatrix}}_{\mathbf{B}_2}y_1[k].\nonumber
\label{eq:err_example}
\end{eqnarray}

To achieve $\lim_{k\to\infty}\mathbf{e}[k]=\mathbf{0}$ regardless of the initial conditions and the trajectories of $x^{(1)}[k]$ and $x^{(2)}[k]$, we require $\mathbf{M}$ to be Schur stable and $\mathbf{B}_1$ and $\mathbf{B}_2$ to be zero.
To obtain $\mathbf{B}_1=\mathbf{0}$, we must set $\alpha_1=\alpha_2=\alpha_3=\alpha$. However, it is impossible to set $\mathbf{B}_2$ to zero due to the non-zero coupling term $\beta$ between the function of interest $x_1[k]$ and the measurement $y_1[k]=x_2[k]$. Based on the unstable dynamics of $y_1[k]$ (see \eqref{eq:example_model}), it follows that no choice of free design parameters at the various nodes can guarantee $\lim_{k\to\infty}{e}_3[k]=\mathbf{0}.$
\end{proof}
\begin{remark}
For the sake of illustration, we considered the observer model given by \eqref{eqn:distobsexample}.
However, it should be noted that given the plant and measurement model \eqref{eq:example_model}, and the network depicted by Figure \ref{fig:example}, adding more free design parameters to the observer structure would not change the result of Proposition \ref{prop:example}.
\end{remark}

To see why Darouach's conditions do not generally hold in a distributed setting, let us take a closer look at the rank condition \eqref{eq:Darouach_rank}. We see that \eqref{eq:Darouach_rank} implies the existence of matrices $\mathbf{M}_1,\mathbf{M}_2,\mathbf{M}_3$ (not necessarily unique) such that $\mathbf{LA}=\mathbf{M}_1\mathbf{L}+\mathbf{M}_2\mathbf{C}+\mathbf{M}_3\mathbf{CA}$. Thus, referring to equations \eqref{eqn:plant}, \eqref{eqn:Obsmodel} and \eqref{eq:funcdef}, we have $\boldsymbol{\psi}[k+1]=\mathbf{M}_1\boldsymbol{\psi}[k]+\mathbf{M}_2{\mathbf{y}}[k]+\mathbf{M}_3\mathbf{y}[k+1]$. Since the dynamics of $\boldsymbol{\psi}[k]$ are coupled to the measurements that are directly available in a centralized setting (and hence require no estimation), it suffices to maintain a dynamic estimator of order equal to the length of the vector $\boldsymbol{\psi}[k]$. However, in a distributed setting, the entire measurement vector $\mathbf{y}[k]$ is no longer accessible at a single sensor node, thereby precluding the direct use of a centralized functional observer design. 

At this stage, it should be pointed out that even if Darouach's conditions are not met, it is still possible to construct minimal order centralized functional observers of order greater than $r$ \cite{trinh}. However, based on the discussion in this section, to isolate the challenges introduced by the distributed setting, we restrict our attention to tuples $(\mathbf{A,C,L})$ that allow for the construction of $r$-th order centralized functional observers and additionally possess certain extra structure (to be discussed later). Note that the notion of a `minimal order distributed functional observer' is not clearly defined: one may require all nodes to maintain observers of the same order and seek to minimize such an order. Alternatively, one may allow the nodes to maintain observers of different orders and seek to minimize the average order (there might be other possible interpretations as well). In this paper, we stick to the former notion and develop a distributed functional estimation strategy that requires all nodes to maintain observers of the same order that is in general smaller the dimension of the state $\mathbf{x}[k]$. However, the problem of defining and subsequently obtaining a minimal order distributed functional observer remains open.

For the problem under consideration, since the dynamics of $\boldsymbol{\psi}[k]$ are coupled with the measurements that are no longer co-located, it becomes necessary for the sensor nodes to maintain estimates of each others' measurements in order to estimate $\boldsymbol{\psi}[k]$.\footnote{This is illustrated by the system considered in Proposition \ref{prop:example} where the function of interest, namely $x^{(1)}[k]$, is coupled to the state measured by node $1$. Hence, nodes that are not immediate neighbors of node $1$ (like node $3$) need to maintain estimates of $y_1[k]$ in order to estimate $x^{(1)}[k]$.} Building on this intuition, we introduce the notion of `\textit{functional leader sets}' in the following section, and investigate how the dynamics of  the functions $\boldsymbol{\psi}[k]$ are coupled to the measurements of such a set of nodes via an appropriately designed similarity transformation.
\section{Functional Leader Sets} Before formally defining a functional leader set, we need to first introduce some terminology. To this end, note that by a row sub-matrix $\bar{\mathbf{C}}_{\mathcal{S}}$ of $\mathbf{C}_{\mathcal{S}}$, we imply that $\bar{\mathbf{C}}_{\mathcal{S}}$ contains a non-empty (not necessarily proper) subset of the rows of $\mathbf{C}_{\mathcal{S}}$, i.e., $\mathcal{R}(\bar{\mathbf{C}}_{\mathcal{S}})\subseteq\mathcal{R}(\mathbf{C}_{\mathcal{S}})$. Consider the following definitions. 
\begin{definition}[Feasible Leader Set]
\label{defn:feasibleleader} A set of nodes $\mathcal{S}\subseteq\mathcal{V}$ is called a feasible leader set if there exists at least one row sub-matrix $\bar{\mathbf{C}}_{\mathcal{S}}$ of $\mathbf{C}_{\mathcal{S}}$ satisfying the following two conditions:
\begin{itemize}
\item[(i)] \begin{equation}
\textrm{rank}\begin{bmatrix}\mathbf{LA}\\ \mathbf{\bar{\mathbf{C}}_{\mathcal{S}}A} \\ \mathbf{L} \\ \mathbf{\bar{\mathbf{C}}_{\mathcal{S}}} \end{bmatrix}=\textrm{rank}\begin{bmatrix}\mathbf{L}\\\mathbf{\bar{\mathbf{C}}_{\mathcal{S}}}\end{bmatrix},
\label{eq:rankcond1}
\end{equation} 

\item[(ii)] 
\begin{equation}
\textrm{rank}\begin{bmatrix}s\begin{bmatrix}\mathbf{L}\\\mathbf{\bar{\mathbf{C}}_{\mathcal{S}}}\end{bmatrix}-\begin{bmatrix}\mathbf{L}\\\mathbf{\bar{\mathbf{C}}_{\mathcal{S}}}\end{bmatrix}\mathbf{A}\\
\mathbf{\bar{\mathbf{C}}_{\mathcal{S}}}
\end{bmatrix}
=\textrm{rank}\begin{bmatrix}\mathbf{L}\\\mathbf{\bar{\mathbf{C}}_{\mathcal{S}}}\end{bmatrix}, \hspace{2mm}\forall s\in \mathbb{C}, \hspace{2mm} |s| \geq 1.
\label{eq:detectcond}
\end{equation}
\end{itemize} 
\end{definition}
\begin{definition}[Minimal Leader Set] A set $\mathcal{S}$ is called a minimal leader set if $\mathcal{S}$ is a feasible leader set and no subset  of $\mathcal{S}$ is a feasible leader set. A feasible leader set $\mathcal{S} $ with $|\mathcal{S}|=1$ is considered to be minimal by default.
\end{definition}

Given a minimal leader set $\mathcal{S}$, if there are several row sub-matrices of ${\mathbf{C}}_{\mathcal{S}}$ that satisfy conditions \eqref{eq:rankcond1} and \eqref{eq:detectcond}, denote the row sub-matrix that produces the lowest rank of ${\begin{bmatrix} \mathbf{L}^{T}&{\bar{\mathbf{C}}_{\mathcal{S}}}^{T}\end{bmatrix}}^{T}$ by $\mathbf{C}_{\mathcal{S}_{min}}$ and the corresponding rank by $r_{\mathcal{S}_{min}}$.
Let the set of all feasible leader sets be denoted by $\mathcal{F}$ and the set of all minimal leader sets be denoted by $\mathcal{M}=\{\mathcal{S}^{(1)},\cdots,\mathcal{S}^{(l)}\}$, where $l=|\mathcal{M}|.$ The tuples characterizing the minimal leader sets are given by $\{(\mathbf{C}_{\mathcal{S}^{(1)}_{min}},r_{\mathcal{S}^{(1)}_{min}}),\cdots,(\mathbf{C}_{\mathcal{S}^{(l)}_{min}},r_{\mathcal{S}^{(l)}_{min}})\}.$ 
\begin{definition}[Functional Leader Set] A set $\mathcal{S}^{(i)}\in\mathcal{M}$ is referred to as a functional leader set if $r_{\mathcal{S}^{(i)}_{min}}\leq r_{\mathcal{S}^{(j)}_{min}} \forall j\in\{1,\cdots,l\}\setminus\{i\}$. 
\end{definition}

Thus, a functional leader set is a minimal leader set that yields the lowest rank on the R.H.S. of equation \eqref{eq:rankcond1} among all minimal leader sets. Given any tuple $(\mathbf{A,C,L})$ described by \eqref{eqn:plant}, \eqref{eqn:Obsmodel}, \eqref{eq:funcdef}, if $\mathcal{F}$ is non-empty, then it is easily seen that $\mathcal{M}$ is also non-empty and hence we are guaranteed the existence of at least one functional leader set. If there are multiple functional leader sets, it suffices to pick any one for our subsequent analysis since all such sets will essentially lead to distributed  functional observers of the same order. Thus, if $\mathcal{F}$ is non-empty, we pick any functional leader set and denote it by ${\mathcal{S}^{\star}}$. For notational simplicity, we denote the tuple characterizing ${\mathcal{S}^{\star}}$ by $(\mathbf{C}^{\star},r^{\star})$.
The nodes in $\mathcal{S}^{\star}$ are referred to as \textit{functional leader nodes} and it will be subsequently shown that such nodes play a key role in solving Problem 1.

\begin{remark} 
\label{rem:functset}
Roughly speaking, it will soon be apparent that any set of sensor nodes belonging to $\mathcal{F}$ (and hence $\mathcal{M}$) can effectively serve as `leaders' in the consensus dynamics for estimating the functions of interest, thereby justifying the proposed terminology. Furthermore, if set $\mathcal{S}^{(i)}\in\mathcal{M}$ is chosen as the leader set, then our design would result in every node maintaining a distributed functional observer of order $r_{\mathcal{S}^{(i)}_{min}}$. The definition of a functional leader set \footnote{Given a tuple $(\mathbf{A,C,L})$ described by \eqref{eqn:plant}, \eqref{eqn:Obsmodel} and \eqref{eq:funcdef}, the design of an algorithm that finds a functional leader set $\mathcal{S}^{\star}$  (provided $\mathcal{F}$ is non-empty), and the subsequent analysis of its complexity, are interesting  avenues of future research.} is thus motivated by the goal of obtaining the distributed functional observer of minimal order among all feasible leader sets.
\end{remark}

Before proceeding further, we illustrate some of the concepts introduced in this section via the following model:
\begin{equation}
\mathbf{A}=\begin{bmatrix}0&2&0\\3&0&0\\0&0&5\end{bmatrix}, \mathbf{C}_1=\begin{bmatrix}0&1&0\\0&0&1\end{bmatrix}, \mathbf{L}=\begin{bmatrix}1&0&0\end{bmatrix}.
\label{eq:illustrate}
\end{equation}
Clearly, $\mathcal{S}=\{1\}$ is a minimal leader set with $\mathbf{C}_1$ satisfying both the rank conditions \eqref{eq:rankcond1} and \eqref{eq:detectcond}. However, these conditions are also satisfied by the row sub-matrix formed by considering just the first row of $\mathbf{C}_1$. While considering the entire $\mathbf{C}_1$ will lead to a distributed functional observer of order $3$, considering only its first row will lead to an observer of order $2$ using our design methodology. Given a minimal leader set $\mathcal{S}$, the foregoing discussion motivates the need to check whether sub-matrices of $\mathbf{C}_{\mathcal{S}}$ satisfy the conditions \eqref{eq:rankcond1} and \eqref{eq:detectcond}.\footnote{Intuitively, we see that the state of interest, namely state $1$, is coupled only to the second state. Hence, the extra information about the third state provided by the second row of $\mathbf{C}_1$ is irrelevant in the present context. Based on this discussion, note that our approach ensures that the order of the proposed distributed functional observer is in general smaller than the dimension of the detectable subspace of the pair $(\mathbf{A,C})$, where $\mathbf{C}$ represents the collective observation matrix.} With $\mathbf{A}$ and $\mathbf{L}$ as described in \eqref{eq:illustrate}, suppose we had $\mathbf{C}_1=\begin{bmatrix}0&1&0\end{bmatrix}$ and $\mathbf{C}_2=\begin{bmatrix}0&0&1\end{bmatrix}$. Then, $\mathcal{S}=\{1,2\}$ would be a feasible (but not minimal) leader set, $\mathcal{S}=\{1\}$ would be a minimal leader set and $\mathcal{S}=\{2\}$ would not be a feasible leader set.

The following property of the functional leader set $\mathcal{S}^{\star}$ will be critical in our subsequent design.\footnote{Clearly, any feasible leader set possesses a similar property; the rationale behind considering the functional leader set in particular is made apparent by Remark \ref{rem:functset}.}
\begin{lemma} 
\label{lemma:funcset}
Given a tuple $(\mathbf{A,C,L})$ described by \eqref{eqn:plant}, \eqref{eqn:Obsmodel} and \eqref{eq:funcdef} such that $\mathcal{F}$ is non-empty, let the functional leader set $\mathcal{S}^{\star}$ be characterized by the tuple $(\mathbf{C}^{\star},r^{\star})$ with $p$ denoting the number of rows of $\mathbf{C}^{\star}$. Then, there exists a similarity transformation matrix $\mathbf{T}$ that brings $\mathbf{(A,C^{\star})}$ to the following form:
\begin{equation}
\bar{\mathbf{A}} =\begin{bmatrix} \mathbf{A}_{D} & \mathbf{0}\\
\mathbf{A}_{E} & \mathbf{A}_{F}\end{bmatrix}, \hspace{1mm}
\bar{\mathbf{C}} =\begin{bmatrix}\mathbf{C}_{D} & \mathbf{0}
\end{bmatrix},
\label{eq:transform1}
\end{equation}
where $\mathbf{A}_{D} \in \mathbb{R}^{r^{\star}\times r^{\star}}$, $\mathbf{C}_{D} \in \mathbb{R}^{p\times r^{\star}}$. Furthermore, the following properties hold: (i) the state vector corresponding to the matrix $\mathbf{A}_{D}$ has the functions of interest $\boldsymbol{\psi}[k]$ as its first $r$ components, and a subset of measurements corresponding to the matrix $\mathbf{C}^{\star}$ as the remaining $r^{\star}-r$ components; and (ii) the pair $(\mathbf{A}_{D},\mathbf{C}_{D})$ is detectable.
\end{lemma}
\begin{proof}
By definition, since $\mathcal{S}^{\star}$ is a feasible leader set, the rank conditions \eqref{eq:rankcond1} and \eqref{eq:detectcond} are satisfied by $\bar{\mathbf{C}}_{\mathcal{S}^{\star}}=\mathbf{C}^{\star}$. In particular, based on the rank condition \eqref{eq:rankcond1}, it is easy to see that $\mathcal{R}({\begin{bmatrix} \mathbf{L}^T&{\mathbf{C}^{\star}}^T\end{bmatrix}}^T)$ is $\mathbf{A}^T$-invariant. Define $\mathbf{\Sigma} \triangleq {\begin{bmatrix} \mathbf{L}^T&{\tilde{{\mathbf{C}}^{\star}}}^T\end{bmatrix}}^T$, where $\tilde{{\mathbf{C}}^{\star}}$ contains all the linearly independent rows of $\mathbf{C}^{\star}$ that are also linearly independent of the rows of $\mathbf{L}$. Noting that $\textrm{rank}\hspace{1mm}\mathbf{\Sigma}=r^{\star}$, it follows that there exists a matrix ${\mathbf{A}}_{D} \in \mathbb{R}^{r^{\star}\times r^{\star}}$ such that
\begin{equation}
\mathbf{\Sigma}^{}\mathbf{A}={\mathbf{A}}_{D}\mathbf{\Sigma}^{}.
\label{eq:statedecoup}
\end{equation}
Let us define a non-singular transformation matrix as $\mathbf{T}\triangleq{\begin{bmatrix}{\mathbf{\Sigma}}^T & {\mathbf{V}}^T\end{bmatrix}}^T$, where the rows of $\mathbf{V} \in {\mathbb{R}}^{(n-r^{\star})\times{n}}$ represent an orthogonal basis for the null space of $\mathbf{\Sigma}^{}$. Using \eqref{eq:statedecoup}, we then conclude that
\begin{equation}
\mathbf{TA}=\begin{bmatrix}\mathbf{\Sigma}^{}\mathbf{A}\\\mathbf{VA}\end{bmatrix}=\begin{bmatrix}\mathbf{A}_{D}\mathbf{\Sigma}^{}\\ \mathbf{VA}\end{bmatrix}=
\begin{bmatrix}\begin{bmatrix}\mathbf{A}_{D}&\mathbf{0}\end{bmatrix}\mathbf{T}\\
\mathbf{VA}{\mathbf{T}}^{-1}\mathbf{T}\end{bmatrix}.
\end{equation}
By partitioning the $(n-r^{\star})\times n$ matrix $\mathbf{VA}{\mathbf{T}}^{-1}$ as $\mathbf{VA}{\mathbf{T}}^{-1}=\begin{bmatrix}\mathbf{A}_{E} & \mathbf{A}_{F}\end{bmatrix}$, where $\mathbf{A}_{E}\in\mathbb{R}^{(n-r^{\star})\times r^{\star}}$ and $\mathbf{A}_{F}\in\mathbb{R}^{(n-r^{\star})\times (n-r^{\star})}$, we further obtain
\begin{equation}
\mathbf{TA}=\begin{bmatrix} \mathbf{A}_{D} & \mathbf{0}\\
\mathbf{A}_{E} & \mathbf{A}_{F}\end{bmatrix}\mathbf{T}.
\label{eq:sim1}
\end{equation}
Next, note that $\mathcal{R}(\mathbf{C}^{\star})\subseteq\mathcal{R}(\mathbf{\Sigma}^{})$. Hence, there exists a matrix $\mathbf{C}_{D} \in \mathbb{R}^{p\times r^{\star}}$ such that 
\begin{equation}
\mathbf{C}^{\star}=\mathbf{C}_{D}\mathbf{\Sigma}^{}.
\label{eq:Cdecoup}
\end{equation}
Noting that $\mathbf{\Sigma}^{}{\mathbf{T}}^{-1}=\begin{bmatrix}\mathbf{I}_{}&\mathbf{0}\end{bmatrix}$, and using \eqref{eq:Cdecoup}, we see that
\begin{equation}
\mathbf{C}^{\star}{\mathbf{T}}^{-1}=\begin{bmatrix}\mathbf{C}_{D} & \mathbf{0}
\end{bmatrix}.
\label{eq:sim2}
\end{equation}
Defining $\bar{\mathbf{A}}\triangleq\mathbf{T}\mathbf{A}{\mathbf{T}}^{-1}$, $\bar{\mathbf{C}}\triangleq\mathbf{C}^{\star}{\mathbf{T}}^{-1}$, and using \eqref{eq:sim1}, \eqref{eq:sim2}, we obtain \eqref{eq:transform1}. It remains to show that \eqref{eq:detectcond} implies detectability of the pair $(\mathbf{A}_{D},\mathbf{C}_{D})$. To this end, note that the unique solutions to \eqref{eq:statedecoup} and \eqref{eq:Cdecoup} are given by $\mathbf{A}_{D}=\mathbf{\Sigma}^{}\mathbf{A}{{\mathbf{\Sigma}}^{}}^{\dagger}$ and $\mathbf{C}_{D}=\mathbf{C}^{\star}{{\mathbf{\Sigma}}^{}}^{\dagger}$, respectively. Based on the PBH test, the pair $(\mathbf{A}_{D},\mathbf{C}_{D})$ is detectable if and only if
\begin{equation}
\textrm{rank}\begin{bmatrix}s\mathbf{I}_{}-{\mathbf{\Sigma}}^{}\mathbf{A}{{\mathbf{\Sigma}}^{}}^{\dagger}\\
\mathbf{C}^{\star}{{\mathbf{\Sigma}}^{}}^{\dagger}
\end{bmatrix}
=r^{\star}, \hspace{2mm}\forall s\in \mathbb{C}, \hspace{2mm} |s| \geq 1.
\label{eq:PBH}
\end{equation}
Since $\begin{bmatrix} {{\mathbf{\Sigma}}^{}}^{\dagger} & \mathbf{I}-{{\mathbf{\Sigma}}^{}}^{\dagger}\mathbf{\Sigma}\end{bmatrix}$ is full row-rank, it follows that
\begin{eqnarray}
\textrm{rank}\begin{bmatrix}s{\mathbf{\Sigma}}^{}-{\mathbf{\Sigma}}^{}\mathbf{A}\\
\mathbf{C}^{\star}
\end{bmatrix}&=&\textrm{rank}\begin{bmatrix}s\mathbf{\Sigma}^{}-\mathbf{\Sigma}^{}\mathbf{A}\\
\mathbf{C}^{\star}
\end{bmatrix}\begin{bmatrix} {{\mathbf{\Sigma}}^{}}^{\dagger} & \mathbf{I}_{}-{{\mathbf{\Sigma}}^{}}^{\dagger}{\mathbf{\Sigma}}^{}\end{bmatrix}\nonumber\\
&=&\textrm{rank}\begin{bmatrix}s\mathbf{I}_{}-\mathbf{\Sigma}\mathbf{A}{\mathbf{\Sigma}}^{\dagger}&\mathbf{0}\\
\label{eq:rank}
\mathbf{C}^{\star}{\mathbf{\Sigma}}^{\dagger}&\mathbf{0}
\end{bmatrix}.
\end{eqnarray}
The last equality follows by noting that the matrix $(\mathbf{I}-{\mathbf{\Sigma}}^{\dagger}\mathbf{\Sigma})$ projects onto the null space of $\mathbf{\Sigma}$, and that $\mathcal{R}(\mathbf{\Sigma}\mathbf{A})$, $\mathcal{R}(\mathbf{C}^{\star})$ are both contained in $\mathcal{R}(\mathbf{\Sigma})$. Finally, combining \eqref{eq:detectcond}, \eqref{eq:PBH} and \eqref{eq:rank} leads to the desired result.
\end{proof}
\begin{remark} Note that  Lemma \ref{lemma:funcset} does not describe a standard detectable decomposition of the pair $(\mathbf{A},\mathbf{C}^{\star})$. In fact, the dimension of the square matrix $\mathbf{A}_{D}$, namely $r^{\star}$, is in general smaller than the dimensions of the detectable subspaces of the pairs $(\mathbf{A},\mathbf{C}^{\star})$, $(\mathbf{A},\mathbf{C}_{\mathcal{S}^{\star}})$ and $(\mathbf{A,C})$.
\end{remark}

Based on \eqref{eq:transform1}, we obtain the following dynamics:
\begin{equation}
\begin{split}
\boldsymbol{\phi}[k+1]=\mathbf{A}_{D}\boldsymbol{\phi}[k],
\bar{y}[k]=\mathbf{C}_{D}\boldsymbol{\phi}[k],
\end{split}
\label{eq:redmodel}
\end{equation}
where $\boldsymbol{\phi}[k]=\begin{bmatrix} \mathbf{I}_{r^{\star}}&\mathbf{0}\end{bmatrix}\mathbf{Tx}[k]=\mathbf{\Sigma}\mathbf{x}[k]$, $\bar{\mathbf{y}}[k]$ represents measurements of the sensor nodes in $\mathcal{S}^{\star}$ corresponding to the matrix $\mathbf{C}^{\star}$, i.e., $\bar{\mathbf{y}}[k]=\mathbf{C}^{\star}\mathbf{x}[k]$, and $\mathbf{A}_{D}$, $\mathbf{C}_{D}$, $\mathbf{\Sigma}={\begin{bmatrix} \mathbf{L}^T&{\tilde{{\mathbf{C}}^{\star}}}^T\end{bmatrix}}^T$ are as described by Lemma \ref{lemma:funcset}. In particular, note that the first $r$ states of the vector $\boldsymbol{\phi}[k]$ represent the functions of interest, namely $\boldsymbol{\psi}[k]$. Effectively, we have converted the distributed functional estimation problem to the problem of designing a full-order distributed state observer for the state $\boldsymbol{\phi}[k]$ described by \eqref{eq:redmodel}. At this stage, any of the distributed state observer approaches outlined in \cite{park,wang,mitra2016} can be employed for estimating $\boldsymbol{\phi}[k]$. In what follows, we develop a  distributed functional observer that guarantees asymptotic reconstruction of $\boldsymbol{\phi}[k]$, and hence $\boldsymbol{\psi}[k]$ at every sensor node, based on the approach adopted in our recent work \cite{mitra2016}. We do so because the approaches described in \cite{wang,park} require certain sensor nodes to maintain observers of order greater than $r^{\star}$, whereas with our method all nodes maintain observers of order $r^{\star}$.
\section{Distributed Functional Observer Design}
\label{sec:Distributedfuncobs}
For clarity of exposition, we make the following assumption. 

\begin{assumption}
The graph $\mathcal{G}$ is strongly connected.\footnote{A strongly connected graph $\mathcal{G}$ contains a directed path from any node $i$ to any other node $j$, where $i,j \in \mathcal{V}$.}
\label{assump:graph}
\end{assumption}

Referring to the dynamics \eqref{eq:redmodel}, our distributed functional estimation strategy can be summarized as: given a node $i$ in $\mathcal{S}^{\star}$, we wish to identify the portion of $\boldsymbol{\phi}[k]$ that is locally detectable at node $i$. Accordingly, node $i$ will maintain a local Luenberger observer for this locally detectable portion and rely on consensus dynamics for estimating the rest of  $\boldsymbol{\phi}[k]$. Given this strategy for the functional leader nodes, it will be established that the nodes in $\mathcal{V}\setminus\mathcal{S}^{\star}$ need to simply run consensus for different components of $\boldsymbol{\phi}[k]$ along spanning trees rooted at the functional leader set $\mathcal{S}^{\star}$. To achieve this objective, we  first use the concept of a \textit{multi-sensor observable canonical decomposition}. Essentially, such a decomposition proceeds as follows: given a list of $M$ indexed sensors, perform an observable canonical decomposition with respect to the first sensor. Next, perform another observable canonical decomposition that reveals the observable subspace of sensor 2 \textit{within the unobservable subspace of sensor 1}, and repeat the process until the last sensor is reached. Summarily, a sequence of $M$ observable canonical decompositions need to be performed, one for each sensor, with the last decomposition revealing the portions of the state space that can and cannot be observed using the cumulative measurements of all the sensors. Without loss of generality, let the sensors in the functional leader set $\mathcal{S}^{\star}$ be indexed as $\mathcal{S}^{\star}=\{1,2,\cdots,M\}$ where $M=|\mathcal{S}^{\star}|$. The next result then readily follows from \cite{mitra2016}.
\begin{proposition}
\label{transformations}
Given the dynamics \eqref{eq:redmodel}, let the matrix $\mathbf{C}_{D}$ be partitioned among the individual sensor nodes of the functional leader set $\mathcal{S}^{\star}$ as $\mathbf{C}_{D}={\begin{bmatrix}\mathbf{C}^T_{D_1} & \cdots & \mathbf{C}^T_{D_M}\end{bmatrix}}^T$.\footnote{Accordingly, $\bar{\mathbf{y}}_{i}[k]=\mathbf{C}_{D_i}\boldsymbol{\phi}[k], \forall{i}\in\mathcal{S}^{\star}$.} Then, there exists a similarity transformation matrix $\mathcal{T}_{D}\in\mathbb{R}^{r^{\star}\times r^{\star}}$ which transforms the pair $(\mathbf{A}_{D},\mathbf{C}_{D})$ to $(\bar{\mathbf{A}}_{D},\bar{\mathbf{C}}_{D})$, such that
\begin{equation}
\begin{split}
\bar{\mathbf{A}}_{D} &= \left[
\begin{array}{c|c|ccc}
\mathbf{A}_{11}  & \multicolumn{4}{c}{\mathbf{0}} \\
\hline
\mathbf{A}_{21}
 & 
\mathbf{A}_{22} & \multicolumn{3}{c}{\mathbf{0}} \\
\cline{2-5}
\vdots & \vdots & \hspace{-5mm} \ddots  & \vdots & \vdots\\
\mathbf{A}_{(M-1)1}& \mathbf{A}_{(M-1)2} \cdots &  \mathbf{A}_{(M-1)(M-1)} & \multicolumn{2}{|c}{\mathbf{0}}\\
\cline{3-5}
\mathbf{A}_{M1}& \mathbf{A}_{M2} \hspace{2mm}\cdots & \mathbf{A}_{M(M-1)} & \multicolumn{1}{|c}{\mathbf{A}_{MM}} & \multicolumn{1}{|c}{\mathbf{0}}\\
\cline{4-5}
\mathbf{A}_1& \hspace{-2mm}\mathbf{A}_2 \hspace{2mm} \cdots & \mathbf{A}_{(M-1)} & \multicolumn{1}{|c}{\mathbf{A}_{M}} & \multicolumn{1}{|c}{\mathbf{A}_{\mathcal{U}}}\\
\end{array}
\right], \\~\\
\bar{\mathbf{C}}_{D} &= \begin{bmatrix} \bar{\mathbf{C}}_{D_1} \\ \bar{\mathbf{C}}_{D_2} \vspace{-1mm} \\  \vdots \\ \bar{\mathbf{C}}_{D_M} \end{bmatrix} = \left[ \begin{array}{ccccc} \mathbf{C}_{11} & \multicolumn{4}{|c}{\mathbf{0}}\\
\hline
\mathbf{C}_{21} & \multicolumn{1}{c}{\mathbf{C}_{22}} & \multicolumn{3}{|c}{\mathbf{0}}\\
\hline
\vdots&\vdots&\vdots&\vdots&\vdots\\
\mathbf{C}_{M1} & \multicolumn{1}{c}{\mathbf{C}_{M2}}  & \cdots  \mathbf{C}_{M(M-1)} & \mathbf{C}_{MM}  & \multicolumn{1}{|c}{\mathbf{0}}\\
 \end{array}
 \right].
\end{split}
\label{eqn:gen_form}
\end{equation}
Furthermore, the following properties hold: (i) the pair $(\mathbf{A}_{ii},\mathbf{C}_{ii})$ is observable $\forall i \in \{1,2, \cdots, M\}$; and (ii) the matrix $\mathbf{A}_{\mathcal{U}}$ describes the dynamics of the unobservable subspace of the pair $(\mathbf{A}_{D},\mathbf{C}_D)$.
\end{proposition}

Referring to Proposition \ref{transformations}, we use the matrix $\mathcal{T}_{D}$ to perform the coordinate transformation $\boldsymbol{\phi}[k]=\mathcal{T}_{D}\mathbf{z}[k]$ to obtain
\begin{eqnarray}
\mathbf{z}[k+1]&=&\bar{\mathbf{A}}_{D}\mathbf{z}[k], \nonumber\\
\bar{\mathbf{y}}_i[k]&=&\bar{\mathbf{C}}_{D_i}\mathbf{z}[k], \quad \forall i \in \{1, \cdots, M\},
\label{eqn:transformed_dyn}
\end{eqnarray}
where $\bar{\mathbf{A}}_{D}={\mathcal{T}_{D}}^{-1}\mathbf{A}_{D}\mathcal{T}_{D}$  and $\bar{\mathbf{C}}_{D_i} = \mathbf{C}_{D_i}\mathcal{T}_{D} = \begin{bmatrix} \mathbf{C}_{i1} & \mathbf{C}_{i2} & \cdots & \mathbf{C}_{i(i-1)} & \mathbf{C}_{ii}  & \multicolumn{1}{|c}{\mathbf{0}}\end{bmatrix}$ are given by (\ref{eqn:gen_form}). The vector $\mathbf{z}[k]$ assumes the structure $\mathbf{z}[k]={\begin{bmatrix}
{{\mathbf{z}}^{(1)}}^{T}[k]&
\cdots&
{{\mathbf{z}}^{(M)}}^{T}[k]&
\mathbf{z}^{T}_{\mathcal{U}}[k]\end{bmatrix}}^{T}$ commensurate with the structure of $\bar{\mathbf{A}}_{D}$ in (\ref{eqn:gen_form}).
We refer to $\mathbf{z}^{(j)}_{}[k] \in {\mathbb{R}}^{o_j}$ as the $j$-th sub-state, and to $\mathbf{z}_{\mathcal{U}}[k]$ as the unobservable sub-state of $\mathbf{z}[k]$ since it represents the unobservable portion of the state with respect to the pair $(\mathbf{A}_{D},\mathbf{C}_{D})$.  The multi-sensor observable canonical decomposition leads to a one-to-one correspondence between a node $j\in\mathcal{S}^{\star}$ and its associated sub-state $\mathbf{z}^{(j)}_{}[k]$. Accordingly, node $j\in\mathcal{S}^{\star}$ is viewed as the source of information of its corresponding sub-state $\mathbf{z}^{(j)}_{}[k]$, and is tasked with the responsibility of estimating $\mathbf{z}^{(j)}_{}[k]$ and leading the consensus dynamics corresponding to $\mathbf{z}^{(j)}_{}[k]$. Thus, for each of the $M$ sub-states, we have a unique source of information within the functional leader set $\mathcal{S}^{\star}$. Our estimation strategy is detailed below.

Based on equations (\ref{eqn:gen_form}) and (\ref{eqn:transformed_dyn}), observe that the dynamics of the $i$-th sub-state are governed by 
\begin{equation}
\begin{split}
\mathbf{z}^{(i)}_{}[k+1]=\mathbf{A}_{ii}\mathbf{z}^{(i)}_{}[k]+\sum_{j=1}^{i-1}\mathbf{A}_{ij}\mathbf{z}^{(j)}_{}[k],\\
\bar{\mathbf{y}}_i[k]=\mathbf{C}_{ii}\mathbf{z}^{(i)}_{}[k]+\sum_{j=1}^{i-1}\mathbf{C}_{ij}\mathbf{z}^{(j)}_{}[k].
\end{split}
\label{eqn:gendyn}
\end{equation}
Note that the unobservable sub-state $\mathbf{z}_{\mathcal{U}}[k]$ evolves as
\begin{equation}
\mathbf{z}_{\mathcal{U}}[k+1]=\mathbf{A}_{\mathcal{U}}\mathbf{z}_{\mathcal{U}}[k]+\sum_{j=1}^{M}\mathbf{A}_{j}\mathbf{z}^{(j)}_{}[k].
\label{eqn:unobs_state}
\end{equation}
Define $\hat{\mathbf{z}}^{(j)}_{i}[k]$ as the estimate of the $j$-th sub-state maintained by the $i$-th node in $\mathcal{V}$. The $i$-th  functional leader adopts the following policy: it uses a Luenberger-style update rule for updating its associated sub-state $\hat{\mathbf{z}}^{(i)}_{i}[k]$, and a consensus based scheme for updating all other sub-states $\hat{\mathbf{z}}^{(j)}_{i}[k]$, where $ j \in \{1, \cdots ,M\} \setminus\{i\}.$ Based on the dynamics (\ref{eqn:gendyn}), the Luenberger observer at node $i\in\mathcal{S}^{\star}$ is constructed as 
\begin{equation}
\begin{split}
\hat{\mathbf{z}}^{(i)}_{i}[k+1]&=\mathbf{A}_{ii}\hat{\mathbf{z}}^{(i)}_{i}[k]+\sum_{j=1}^{i-1}\mathbf{A}_{ij}\hat{\mathbf{z}}^{(j)}_{i}[k]\\
&\hspace{2mm}+\mathbf{G}_i(\bar{\mathbf{y}}_i[k]-(\mathbf{C}_{ii}\hat{\mathbf{z}}^{(i)}_{i}[k]+\sum_{j=1}^{i-1}\mathbf{C}_{ij}\hat{\mathbf{z}}^{(j)}_{i}[k])),\\
\end{split}
\label{eqn:luen}
\end{equation}
where $\mathbf{G}_i \in {\mathbb{R}}^{o_i \times t_i}$ is a gain matrix which needs to be designed ($\mathbf{z}^{(i)}[k] \in {\mathbb{R}}^{o_i}$ and $\bar{\mathbf{y}}_i[k]\in\mathbb{R}^{t_i}$).
For estimation of the $j$-th sub-state, where $j \in \{1, \cdots ,M\}\setminus\{i\}$, the $i$-th node in $\mathcal{S}^{\star}$ employs the following consensus dynamics
\begin{equation}
\hat{\mathbf{z}}^{(j)}_{i}[k+1]=\underbrace{\mathbf{A}_{jj}\sum_{l\in\mathcal{N}_i}w^{j}_{il}\hat{\mathbf{z}}^{(j)}_{l}[k]}_{\hbox{consensus term}}+\underbrace{\sum_{l=1}^{j-1}\mathbf{A}_{jl}\hat{\mathbf{z}}^{(l)}_{i}[k]}_{\hbox{coupling term}},
\label{eqn:consensus}
\end{equation}
where $w^{j}_{il}$ is the weight the $i$-th node associates with the $l$-th node (where $l\in\mathcal{N}_i$) for the estimation of the $j$-th sub-state. Each node in $\mathcal{V}\setminus\mathcal{S}^{\star}$ runs a consensus rule identical to \eqref{eqn:consensus}, but for every sub-state $j\in\{1,\cdots,M\}$. The consensus weights are non-negative and satisfy
\begin{equation}
\sum_{l\in\mathcal{N}_i}w^{j}_{il}=1, \quad \forall j \in \{1, \cdots ,M\}.
\label{eqn:stochasticity}
\end{equation}
 Let $\hat{\mathbf{z}}_{i\mathcal{U}}[k]$ denote the estimate of the unobservable sub-state $\mathbf{z}_{\mathcal{U}}[k]$ maintained by the $i$-th node in $\mathcal{V}$. Mimicking equation (\ref{eqn:unobs_state}), each node $i\in\mathcal{V}$ updates $\hat{\mathbf{z}}_{i\mathcal{U}}[k]$ as
\begin{equation}
\hat{\mathbf{z}}_{i\mathcal{U}}[k+1]=\mathbf{A}_{\mathcal{U}}\hat{\mathbf{z}}_{i\mathcal{U}}[k]+\sum_{j=1}^{M}\mathbf{A}_{j}\hat{\mathbf{z}}^{(j)}_{i}[k].
\label{eqn:unobsestimate}
\end{equation}
Summarily, the update equations (\ref{eqn:luen}), (\ref{eqn:consensus}) and (\ref{eqn:unobsestimate}) constitute the proposed distributed functional observer. Important to note is the fact that each of the $M$ functional leader nodes in $\mathcal{S}^{\star}$ maintain a Luenberger observer for estimating their corresponding sub-state and use consensus for the remaining $M-1$ sub-states of the vector $\mathbf{z}[k]$, whereas each non-leader node in $\mathcal{V}\setminus\mathcal{S}^{\star}$ runs consensus for every sub-state. All nodes run \eqref{eqn:unobsestimate} since there is no leader corresponding to the unobservable sub-state $\mathbf{z}_{\mathcal{U}}[k]$.
\section{Main Result}
In this section, we present our main result concerning the convergence analysis of the distributed functional observer developed in Section \ref{sec:Distributedfuncobs}.
To this end, we require the following Lemma; since the proof of this result essentially follows the same steps as \cite[Theorem 1]{mitra2016}, we skip minor details and present a sketch of the main idea.
\begin{lemma}
\label{lemma_main}
Consider the dynamics \eqref{eq:redmodel} where the pair $(\mathbf{A}_D,\mathbf{C}_D)$ is detectable, and let the communication graph $\mathcal{G}$ be strongly connected. Then, for each node $i \in \mathcal{S}^{\star}= \{1,2, \cdots, M\}$, there exists a choice of observer gain matrix $\mathbf{G}_i$, and consensus weights $w^{j}_{il}$, $j \in \{1,2, \cdots, M\}\setminus\{i\}$, $l \in \mathcal{N}_i$, and for each node $i\in\mathcal{V}\setminus\mathcal{S}^{\star}$, there exists a choice of consensus weights $w^{j}_{il}$, $j \in \{1,2, \cdots, M\}$, $l\in\mathcal{N}_i$, such that the update equations (\ref{eqn:luen}), (\ref{eqn:consensus}) and (\ref{eqn:unobsestimate}) guarantee asymptotic reconstruction of $\boldsymbol{\phi}[k]$ at each node in $\mathcal{V}$.
\end{lemma}
\begin{proof} (\textit{Sketch}) Define $\mathbf{e}^{(j)}_{i}[k] \triangleq \hat{\mathbf{z}}^{(j)}_{i}[k] - \mathbf{z}^{(j)}_{}[k]$ as the error in estimation of the $j$-th sub-state of $\mathbf{z}[k]$ by the $i$-th node in $\mathcal{V}$. Let the index set $\{1,k_1, k_2, \cdots, k_{N-1}\}$ represent a topological ordering consistent with a spanning tree rooted at node 1 (the source of information/leader for sub-state 1); such a tree always exists since $\mathcal{G}$ is strongly connected. Consider the composite error vector (permuted to match the aforementioned topological ordering) for the estimation of sub-state $1$:
$\bar{\mathbf{E}}^{(1)}_{}[k]={\begin{bmatrix}{\mathbf{e}^{(1)}_{1}[k]}^{T}&
{\mathbf{e}^{(1)}_{k_1}[k]}^{T}&\cdots&{\mathbf{e}^{(1)}_{k_{N-1}}[k]}^{T}
\end{bmatrix}}^{T}=\begin{bmatrix}{\mathbf{e}^{(1)}_{1}[k]}^{T}&{{\tilde{\mathbf{E}}}^{(1)}[k]}^{T}\end{bmatrix}$.
Based on \eqref{eqn:gendyn}, \eqref{eqn:luen} and \eqref{eqn:consensus}, we obtain the following error dynamics
\begin{equation}
\left[\begin{array}{c}\mathbf{e}^{(1)}_{1}[k+1]\vspace{2mm}\\
\tilde{\mathbf{E}}^{(1)}_{}[k+1]\end{array}\right]=
\underbrace{\left[\begin{array}{cc} (\mathbf{A}_{11}-\mathbf{G}_{1}\mathbf{C}_{11}) & \mathbf{0} \\
\mathbf{W}^{1}_{21} \otimes \mathbf{A}_{11} & \mathbf{W}^{1}_{22} \otimes \mathbf{A}_{11} \end{array}\right]}_{\mathbf{M}_1}\left[\begin{array}{c}\mathbf{e}^{(1)}_{1}[k]\vspace{2mm}\\
\tilde{\mathbf{E}}^{(1)}_{}[k]\end{array}\right],
\label{eq:errsub1}
\end{equation}
where the matrix $\mathbf{W}^{1}=\begin{bmatrix}\mathbf{W}^{1}_{21} & \mathbf{W}^{1}_{22}\end{bmatrix}$ contains consensus weights defined by (\ref{eqn:consensus}). Notice that $sp(\mathbf{M}_1)=sp(\mathbf{A}_{11}-\mathbf{G}_{1}\mathbf{C}_{11}) \cup sp(\mathbf{W}^{1}_{22} \otimes \mathbf{A}_{11})$. Since the pair $(\mathbf{A}_{11}, \mathbf{C}_{11})$ is observable by construction, one can always find a $\mathbf{G}_1$ that stabilizes $(\mathbf{A}_{11}-\mathbf{G}_1\mathbf{C}_{11})$. Next, we impose the constraint that for the estimation of sub-state 1, non-zero consensus weights are assigned to only the branches of the spanning tree that is rooted at node $1$ and is consistent with the ordering $\{1,k_1, k_2, \cdots, k_{N-1}\}$, i.e., a node listens to only its parent in such a tree. In this way, $\mathbf{W}^1_{22}$ becomes lower triangular with eigenvalues equal to zero. Designing the consensus weights in such a way (without violating \eqref{eqn:stochasticity})    ensures that $\mathbf{W}^{1}_{22} \otimes \mathbf{A}_{11}$, and hence $\mathbf{M}_1$, are both Schur stable, implying $\lim_{k\to\infty} \bar{\mathbf{E}}^{(1)}_{}[k]=\mathbf{0}$; i.e., all nodes in $\mathcal{V}$ can asymptotically estimate sub-state $1$. 

Generalizing the previous analysis, for the estimation of sub-state $j\in\{1,\cdots,M\}$, we impose that a node in $\mathcal{V}$ assigns a non-zero consensus weight to only its parent in the spanning tree rooted at node $j\in\mathcal{S}^{\star}$ (node $j$ acts as the source of information for sub-state $j$). Inducting on the sub-state  number $j$ and using Input to State Stability (ISS) then ensures that every node in $\mathcal{V}$ can asymptotically estimate each of the $M$ sub-states of $\mathbf{z}[k]$. Finally, note that detectability of the pair $(\mathbf{A}_D,\mathbf{C}_D)$ implies that the matrix  $\mathbf{A}_{\mathcal{U}}$ featuring in equations \eqref{eqn:unobs_state} and \eqref{eqn:unobsestimate} is Schur stable. Based on this fact and our previous analysis, it is easy to see that each node in $\mathcal{V}$ can asymptotically estimate $\mathbf{z}_{\mathcal{U}}[k]$, and hence $\mathbf{z}[k]$ and $\boldsymbol{\phi}[k]$ where $\boldsymbol{\phi}[k]=\mathcal{T}_{D}\mathbf{z}[k]$.
\end{proof}
The following is the main result of the paper.
\begin{theorem}
\label{thm:main} Given a tuple $(\mathbf{A,C,L})$ described by equations \eqref{eqn:plant}, \eqref{eqn:Obsmodel}, and \eqref{eq:funcdef}, and a strongly connected communication graph $\mathcal{G}$ satisfying Assumption \ref{assump:graph}, let the feasible leader set $\mathcal{F}$ described by Definition \ref{defn:feasibleleader} be non-empty. Then, the proposed distributed functional observer described by the update equations \eqref{eqn:luen}, \eqref{eqn:consensus}, and \eqref{eqn:unobsestimate}  
solves Problem 1.
\end{theorem}
\begin{proof}
Since $\mathcal{F}$ is non-empty, there exists a functional leader set $\mathcal{S}^{\star}$. Based on the property of $\mathcal{S}^{\star}$ described by Lemma \ref{lemma:funcset}, the pair $(\mathbf{A}_{D},\mathbf{C}_{D})$ governing the dynamics of $\boldsymbol{\phi}[k]$ (see \eqref{eq:redmodel}) is detectable. Since $\mathcal{G}$ is strongly connected, it then follows from Lemma \ref{lemma_main} that each  node in $\mathcal{V}$ can asymptotically estimate $\boldsymbol{\phi}[k]$. Noting that the desired functions $\boldsymbol{\psi}[k]$ satisfy $\boldsymbol{\psi}[k]=\begin{bmatrix} \mathbf{I}_{r}&\mathbf{0}\end{bmatrix}\boldsymbol{\phi}[k]$ completes the proof.
\end{proof}
\begin{remark} \label{rem:order} Note that the order of the proposed distributed functional observer is $r^{\star}$ where $r\leq r^{\star} \leq d$ in general, with $r=\textrm{rank}\hspace{1mm}\mathbf{L}$ and $d$ equal to the dimension of the detectable subspace of the pair $(\mathbf{A},\mathbf{C})$ given by \eqref{eqn:plant} and \eqref{eqn:Obsmodel}.\footnote{Note that if $\mathcal{F}$ is non-empty, then a centralized functional observer of order $r$ can always be constructed.} The fact that $r^{\star}\geq r$ follows from the discussion in Section \ref{sec:motivation}. For the special case when $\mathcal{R}(\mathbf{C}^{\star})\subseteq\mathcal{R}(\mathbf{L})$, we have $r^{\star}=r$. Further, when the tuple $(\mathbf{A,C,L})$ is `functionally observable' \cite{func1}, i.e., when the functions of interest are linear combinations of only the observable states of $(\mathbf{A,C})$, it is easily seen that the order of the proposed observer is no greater than the dimension of the observable subspace of $(\mathbf{A,C})$.
\end{remark}  
\begin{remark} Note that when $\mathbf{L}=\mathbf{I}_{n}$, the rank condition \eqref{eq:rankcond1} is trivially satisfied whereas \eqref{eq:detectcond} boils down to the existence of a set of nodes $\mathcal{S}\in\mathcal{V}$ such that the pair $(\mathbf{A},\mathbf{C}_{\mathcal{S}})$ is detectable. For a strongly connected graph $\mathcal{G}$, it was shown in \cite{mitra2016,wang,park} that the necessary and sufficient condition for  distributed state estimation is the detectability of the pair $(\mathbf{A,C})$. Thus, it is apparent that for a strongly connected graph $\mathcal{G}$, the sufficient condition presented in this paper for the construction of a distributed functional observer, namely that the feasible leader set $\mathcal{F}$ is non-empty, is in fact a generalization of the aforementioned necessary and sufficient condition for distributed state estimation.
\end{remark}

\section{Conclusion}
We studied the problem of designing distributed functional observers for LTI systems. Our work was motivated by the observation that existing results/techniques for designing centralized functional observers are not directly applicable in a distributed setting. To solve the problem considered in this paper, we introduced the notion of `functional leader nodes' and showed that such nodes play a key role in our proposed functional estimation strategy. We established that under certain conditions on the system dynamics and network structure, our method guarantees asymptotic reconstruction of the functions of interest at every sensor node. 

By effectively transforming the distributed functional estimation problem into a distributed state estimation problem, the technique developed in this paper presents various opportunities for extensions. For example, in the case where some nodes in the network are malicious or faulty, one can potentially apply the algorithm from our recent work on secure distributed observers \cite{mitraCDC}.

There are several interesting directions of future research. As pointed out earlier, the problem of defining and subsequently designing a minimal order distributed functional observer remains open. Furthermore, as a generalization of the problem studied in this paper, one might consider a scenario where different clusters of nodes are interested in estimating different sets of functions.

\bibliographystyle{unsrt} \bibliography{refs}
\end{document}